\documentclass{lms}

\usepackage{amsmath,amsfonts, amssymb}
\usepackage{graphicx}
\usepackage{xcolor}
\usepackage{enumerate}
\usepackage{cite}
\usepackage{tabu}

\usepackage{fullpage}
\usepackage[active]{srcltx}

\usepackage{changepage}

\usepackage{chngcntr}

\usepackage{multibib}

\usepackage{url}
\usepackage{hyperref}
\hypersetup{
	colorlinks,
	linkcolor={red!50!black},
	citecolor={blue!50!black},
	urlcolor={blue!80!black}
}  

\newtheorem{theorem}{Theorem}[section]

\newtheorem{proposition}[theorem]{Proposition}
\newtheorem{lemma}[theorem]{Lemma}

\newnumbered{question}{Question}
\newnumbered{definition}{Definition}
\newnumbered{remark}{Remark}

\newcommand{\nbhd}{\ensuremath{\mathcal{N}}}
\newcommand{\bdry}{\ensuremath{\partial}}
\newcommand{\cut}{\ensuremath{\backslash}}
\newcommand{\nil}{\ensuremath{\emptyset}}

\newcommand{\R}{\ensuremath{{\mathbb R}}}

\newcommand{\calR}{\ensuremath{{\mathcal R}}}
\newcommand{\calS}{\ensuremath{{\mathcal S}}}
\newcommand{\MN}{\ensuremath{\mathcal{MN}}}
\newcommand{\LR}{\ensuremath{\mathcal{LR}}}

\newcommand{\HFK}{\ensuremath{{\widehat{\rm HFK}}}}
\DeclareMathOperator{\rank}{rk}

\newcommand*{\qed}{\null\nobreak\hfill\ensuremath{\square}}%

\title[The Morse-Novikov number of knots]{The Morse-Novikov number of knots under connected sum and cabling.}

\author{Kenneth L.\ Baker}
\classno{57M25, 57M27 (primary)}

\extraline{
This work was partially supported by a grant from the Simons Foundation (\#523883 to Kenneth L.\ Baker).
}

\begin{document}

\maketitle

\begin{abstract}
We show the Morse-Novikov number of knots in $S^3$ is additive under connected sum and unchanged by cabling.
\end{abstract}

\section{Introduction}

Given an oriented link $L$ in $S^3$, the {\em Morse-Novikov number} of $L$ is the count $\MN(L)$ of the minimum number of critical points among regular Morse functions $f \colon S^3 - L \to S^1$, see \cite{PRW} and \cite[Definition 14.6.2]{Pajitnov-CMT}.   Pajitnov attributes to M.\ Boileau and C.\ Weber the question of whether the Morse-Novikov number is additive on the connected sum of oriented knots, see the beginning of  \cite[Section 5]{Pajitnov-tunnelMN} and the end of \cite[Section 14.6.2]{Pajitnov-CMT}.   We show that it is.

\begin{theorem}\label{thm:main}
	The Morse-Novikov number is additive:  If $K=K_a \# K_b$ is a connected sum of two oriented knots $K_a$ and $K_b$ in $S^3$, then
	\[\MN(K) = \MN(K_a) + \MN(K_b).\]
\end{theorem}

Instead of working with circle-valued Morse functions directly, we use the handle-theoretic interpretation of the Morse-Novikov number presented in \cite[\S3]{Goda} and \cite[\S2]{GP-Twisted}  that enables the use of techniques from the theory of Heegaard splittings.  This approach is rooted in Goda's work on handle numbers of sutured manifolds and Seifert surfaces \cite{Goda-SuturedManifold, Goda-Handlenumber} and  Manjarrez-Guti\'errez's work on generalized circular Heegaard splittings and circular thin decompositions \cite{MG-CTP}.  Notably, in \cite[Theorem 1.1]{MG-additivity} Manjarrez-Guti\'errez uses this approach to prove Theorem~\ref{thm:main} for the special class of {\em a-small} knots (i.e.\ knots with no closed essential surface disjoint from a Seifert surface) by using a key proposition about the positioning of the summing annulus for a connected sum with respect to a circular (locally) thin decompositions of the knot exterior which she established with Eudave-Mu\~noz \cite[Proposition~5.1]{EMMGcircularthin}.  In essence, we manage to avoid this a-small hypothesis in our proof of Theorem~\ref{thm:main} by paying closer attention to the behavior of counts of handles in generalized circular Heegaard splittings under the operations of weak reductions and amalgamations.

In preparation for \cite[Theorem 1.1]{MG-additivity}, Manjarrez-Guti\'errez shows that the handle number of an a-small knot is realized by the handle number of an incompressible Seifert surface, \cite[Theorem 4.3]{MG-additivity}.  On our way to Theorem~\ref{thm:main} we prove the analogous Lemma~\ref{lem:incompseifertsfce} below which removes the a-small hypothesis.  Our {\em handle number} definitions are given in Definition~\ref{def:handlenumber}.  

\begin{lemma}
\label{lem:incompseifertsfce}
A knot $K \subset S^3$ has an incompressible Seifert surface $R$ such that $h(K)=h(R)$.
\end{lemma}

Finally, we observe that \cite[Proposition 5.1]{EMMGcircularthin} directly generalizes to address cabling annuli.  Consequently, our proof of Theorem~\ref{thm:main} adapts to show that handle number is unaffected by cabling.
\begin{theorem}
	\label{thm:cable}
	Let $K_{p,q}$ be the $(p,q)$--cable of the knot $K$ for coprime integers $p,q$ with $p>0$.   Then $h(K_{p,q})=h(K)$ and hence $\MN(K_{p,q})=\MN(K)$.
	\end{theorem}

For the ease of exposition, we content ourselves with focusing upon knots in $S^3$.  However,  Lemma~\ref{lem:incompseifertsfce}, Theorem~\ref{thm:main}, and Theorem~\ref{thm:cable} can be immediately generalized to null-homologous knots in rational homology spheres.   With a little more work they should also generalize to rationally null-homologous knots in other orientable $3$--manifolds.  

\subsection{Proof sketches}
Let us quickly sketch the proofs of Lemma~\ref{lem:incompseifertsfce} and Theorem~\ref{thm:main} for readers already familiar with the notions of (circular) generalized Heegaard splittings.   Their full proofs are given in \S\ref{sec:proofs}. As the proof of Theorem~\ref{thm:cable} is quite similar to that of Theorem~\ref{thm:main}, we wait to address it in \S\ref{sec:cables} after setting up some groundwork for satellites in \S\ref{sec:satellite}.
Prior to the two sketches, there are a few things worth clarifying now which we will address more fully in  \S\ref{sec:prelim}.  

 We allow our compression bodies to have vertical boundary so that they may be regarded as a kind of  sutured manifold without toroidal sutures.  For us, a generalized Heegaard splitting consists of a sutured manifold $M$ (with positive and negative subsurfaces $R_+$ and $R_-$ of $\bdry M$, with annular sutures between their boundaries, and possibly with toroidal sutures) and a disjoint pair of properly embedded (possibly disconnected) ``thin'' surface $\calR$ and ``thick'' surface $\calS$ that decompose $M$ into compression bodies.   The  positive boundaries of connected compression bodies are components of $\calS$, while the negative boundaries are unions of components of $\calR \cup R_+ \cup R_-$ (satisfying some conditions on $R_+$ and $R_-$).  We also frequently suppress the term circular as it is implied when $M$ has a toroidal suture.

The {\em handle number} $h(W)$ of a compression body $W$ is the minimum number of $0$-- and $1$-- handles used in its construction.
We define its handle index $j(W)$  to be the number of $1$--handles minus the number of $0$--handles used in its construction.  (Dually, it is the number of $2$--handles minus the number of $3$--handles.)  This turns out to be half of the handle index $J$ introduced by Scharlemann-Schultens \cite{SS-annuli,SS-comparing}.   Importantly, the handle index $j$ agrees with the handle number $h$ when the compression body has no handlebody component.  We extend both the handle index and handle number to generalized Heegaard splittings by summing over the compression bodies in the decomposition. (Our use of handle number differs from Goda's at this point.)   The driving observation is that, because weak reductions and amalgamations of generalized Heegaard splittings neither introduce nor cancel $1$-- \& $2$--handle pairs, the handle index is unchanged by these operations.

\begin{proof}[sketch of Lemma~\ref{lem:incompseifertsfce}]
Consider a circular Heegaard splitting $(M,R,S)$ of the exterior $M$ of the knot $K$ with compressible Seifert surface $R$ where $h(K)=h(R)=h(M,R,S)$.  Since the splitting contains no handlebodies, $h(M,R,S)=j(M,R,S)$. Since $R$ is compressible, the splitting is weakly reducible.  A maximal weak reduction along $S$ produces a generalized Heegaard splitting $(M, R_1 \cup R_2, S_1 \cup S_2)$ where $R_1 = R$ and $R_2$ is a potentially disconnected surface that contains an incompressible Seifert surface $R_2^0$.  Then amalgamations along $R \cup R_2 - R_2^0$ produce a circular Heegaard splitting $(M, R_2^0, S'')$ without handlebodies.  Thus $j(M,R,S) = j(M, R_2^0,S'') = h(M,R_2^0,S'') \geq h(R_2^0) \geq h(K)$.  Hence $h(R_2^0) = h(K)$.  See Figure~\ref{fig:incompseifertsfce} for a schematic of the weak reduction and amalgamations.
\end{proof}

\begin{proof}[sketch of Theorem~\ref{thm:main}]
We show $h(K_a \# K_b) \geq h(K_a) + h(K_b)$ as the other inequality is straightforward.   Let $M$ be the exterior of $K$ and let $Q$ be the annulus that decomposes $M$ into the exteriors $M_a$ and $M_b$ of the constituent knots $K_a$ and $K_b$.   

By Lemma~\ref{lem:incompseifertsfce}, there is an incompressible Seifert surface $R$ for $K=K_a\#K_b$ such that $h(K)=h(R)$.  Hence there is a circular Heegaard splitting $(M,R,S)$ realizing this handle number.  A maximal iterated weak reduction of this splitting yields a locally thin generalized Heegaard splitting $(M, \calR, \calS)$ in which $R$ is a component of $\calR$. Thus 
\[h(K) = h(R) = h(M,R,S) = j(M,R,S) = j(M, \calR, \calS).\]
By \cite[Proposition~5.1]{EMMGcircularthin}, since $(M, \calR, \calS)$ is locally thin, we may isotop $Q$ so that any intersection with a compression body of the splitting is a product disk.  Then $Q$ chops this splitting of $M$ into splittings $(M_i, \calR_i, \calS_i)$ of $M_i$ for $i=a,b$.  Furthermore
\[j(M, \calR, \calS) = j(M_a, \calR_a, \calS_a) + j(M_b, \calR_b, \calS_b).\]
For each $i=a,b$, a sequence of amalgamations of $(M_i, \calR_i, \calS_i)$ produces a circular Heegaard splitting $(M_i, R_i, S_i)$ with no handlebodies where $R_i$ is a Seifert surface for $K_i$.  Hence
\[j(M_i, \calR_i, \calS_i) = j(M_i, R_i, S_i) = h(M_i, R_i, S_i) \geq h(R_i) \geq h(K_i)\]
and the result follows.
\end{proof}

\subsection{Questions}
Lemma~\ref{lem:incompseifertsfce} prompts the following question.

\begin{question}
	Is the handle number of a knot always realized by a minimal genus Seifert surface?
\end{question}

While Goda shows this is so for all small crossing knots, he also reveals the challenge with addressing this question:  
there are knots with multiple minimal genus Seifert surfaces that do not all realize the handle number of the knot. See \cite{Goda-Handlenumber}.

\bigskip
\begin{center}
	{---\quad$\bullet$\quad---}
\end{center}
\bigskip

One may also wonder whether Theorem~\ref{thm:main} might be proven more simply, albeit possibly more indirectly, by expressing the Morse-Novikov number in terms of other established knot invariants.  
\begin{question}\label{ques:other}
Can the Morse-Novikov number of a knot be expressed in terms of other established knot invariants? 
\end{question}

While additivity under connected sum is not an uncommon property among knot invariants, the detection of fiberedness is rather exceptional.  For example, while the log of the leading coefficient of the  Alexander polynomial is additive under connected sums of knots, it fails to detect fibered knots. (Indeed, the Alexander polynomial is multiplicative under connected sums and many non-fibered knots have monic Alexander polynomials.) Nevertheless, the log of the rank of the knot Floer Homology of a knot $K$ in the highest non-zero grading, $\LR(K) = \log \rank{\HFK(K,g(K))}$, is both additive on connected sums \cite[Theorem~7.1]{OzSz-HolDiskKnotInvt} and equals zero precisely for fibered knots \cite{Ghiggini-fibered, Ni-fibered}.  However, $\LR$ is distinct from $\MN$; neither is a function of the other.  For alternating knots, $\rank{\HFK(K,g(K))}$ equals the coefficient of the maximal degree term of the Alexander polynomial \cite{OzSz-Alternating}. So $\LR$ varies greatly among the small crossing knots.  However all non-fibered knots of crossing number at most 10 have $\MN = 2$ \cite[\S7]{Goda-Handlenumber}.  Furthermore, the non-fibered twist knots all have $\MN=2$ via \cite{Goda-SuturedManifold} while $\LR$ increases with their twisting (as can be calculated from their Alexander polynomials since they are alternating knots).  

\begin{question}
What knot invariants detect fibered knots and are additive under connected sum?  Which are also unaffected by cabling?
\end{question}

See \cite{taylortomova} for related work on knot invariants that detect the unknot and are additive under connected sum.
Note that $\MN$ does not detect the unknot as it is $0$ for all fibered knots.

\medskip
\begin{center}
	{---\quad$\bullet$\quad---}
\end{center}
\bigskip

As a connected sums and cables are special kinds of satellite knots, one may wonder how handle number behaves under more general satellite operations.   We refer the reader to \S\ref{sec:satellite} for the relevant definitions. 

\begin{question}\label{ques:satellite}
	Let $P$ be a pattern with non-zero winding number, and let $K$ be a knot in $S^3$.  
	Does $h(P(K)) = h(P) + h(K)$?
\end{question}

In Lemma~\ref{lem:satellite} we show $h(P(K)) \leq h(P)+h(K)$.   Part of the difficulty in establishing the equality in general is dealing with how an essential torus may intersect a (locally thin circular generalized) Heegaard splitting. For starters, see \cite{SS-annuli} and \cite{thompson-toriandheegaardsplittings}.  We suspect Question~\ref{ques:satellite} has a negative answer for certain knots and patterns.

Note that in \S\ref{sec:satellite} we only define handle numbers for patterns with non-zero winding number. Indeed, our definitions of $h(P)$ and $h(K)$ are probably too restrictive for a satisfactory understanding of $h(P(K))$ and Question~\ref{ques:satellite} in general without constraints on winding number.


\section{Compression bodies and Heegaard splittings}\label{sec:prelim}
Here we recall fundamental notions in the theory of Heegaard splittings and its generalizations.  As there is a bit of variation in the literature, this allows us to also establish our terminology and notation. We refer the reader to \cite{CG-RHS, Goda-Handlenumber, Goda-SuturedManifold, Goda} for the basic elements of our approach to compression bodies and Heegaard splittings and to \cite{ST-TP3M,Scharlemann-HSo3M,SSS-LectureNotes,MG-CTP,EMMGcircularthin} for the core ideas of generalized Heegaard splittings, circular Heegaard splittings, and  circular generalized Heegaard splittings.  For our purposes in this article, we take care to clarify the operations of weak reductions and amalgamations and refer the reader to \cite{schultens-SfcexS1,SchultensGraph,MG-additivity, lackenby-HGalg} for further discussions of these operations.  

There are a few things to keep in mind when reading the literature.  Some authors require compression bodies to be connected.  We allow them to be disconnected in general but primarily give attention to connected compression bodies.  Some authors view compression bodies as cobordisms between closed surfaces.  We however need them to be cobordisms rel-$\bdry$ between surfaces that may have boundary.  To a great extent these distinctions don't impact the results, though when they do the appropriate modification is typically straightforward.	

Thoughout we will restrict ourselves to only considering irreducible manifolds so that any embedded sphere bounds a $3$--ball. For notation, the result $Y \cut X$ of chopping a manifold $Y$ along a submanifold $X$ is the closure of $Y-X$ in the path metric.  

\subsection{Compression bodies} 

\begin{definition}[{\em Sutured manifolds}]
Introduced by Gabai \cite{Gabai-SM}, a {\em sutured manifold} is a compact oriented $3$--manifold $M$ together with a disjoint pair of subsurfaces $R_+$ and $R_-$ of $\bdry M$  such that $\bdry M \cut (R_+ \cup R_-)$ is a collection of annuli and tori where these annuli join $\bdry R_+$ to $\bdry R_-$.  These complementary annuli and tori are called the {\em sutures}.  The orientation on $R_+$ is taken consistent with the boundary orientation of $\bdry M$ while $R_-$ is oriented oppositely.  We will further assume throughout that no component of $R_+$ or $R_-$ is a sphere.  
\end{definition}

\begin{definition}[{\em Compression bodies}]
	Following \cite{CG-RHS}, a {\em compression body} $W$ is a cobordism rel-$\bdry$ between orientable surfaces $\bdry_+ W$ and $\bdry_- W$ that may be formed as follows: there is a non-empty compact orientable surface $S$ (without sphere components) 
	such that $W$ is assembled from the product $S \times[-1,+1]$ by attaching $2$--handles  to $S \times \{-1\}$ and then $3$--handles to any resulting sphere boundary components meeting $S \times \{-1\}$ that are disjoint from $\bdry S \times [-1,+1]$.  The collection of annuli $\bdry_v W = \bdry S \times [-1,+1]$ is the {\em vertical boundary}  of $W$, and the complement in $\bdry W$ of its interior are the surfaces $\bdry_+ W = S \times \{+1\}$ and $\bdry_- W$. 
	Note that
	\begin{enumerate}
		\item a compression body need not be connected,
		\item $\bdry_+ W$ is necessarily non-empty, and
		\item $\bdry_- W$ has no sphere components. 
	\end{enumerate}
	
	Dually, we may view a compression body $W$ as formed as follows: there is an orientable surface $F$ (without sphere components, and possibly $F = \nil$) such that $W$ is assembled from the disjoint union of the product $F \times [-1,+1]$ and a collection of $0$--handles by attaching a collection of $1$--handles to $F \times \{+1\}$ and the boundaries of these $0$--handles. (Furthermore, every $0$--handle has a $1$--handle attached to it.)  Here $\bdry_v W = \bdry F \times [-1,+1]$, and the complement in $\bdry W$ of its interior are the surfaces $\bdry_-W = F\times \{-1\}$, and $\bdry_+W$.  
	If $W$ is connected, then the components of $F \times [-1,+1]$ and the $0$--handles are all joined by a sequence of $1$--handles.
	Duality exchanges the $0$-- and $3$--handles and the $1$-- and $2$--handles.

 A {\em trivial} compression body $W$ is a product $W = S \times [-1,+1]$ so that $\bdry_+ W = S \times \{+1\}$. 
 A connected compression body $W$ with $\bdry_- W = \nil$ is a {\em handlebody}. 	
\end{definition}

\begin{remark}
For comparison, Bonahon introduced compression bodies $W$ for which $\bdry_+W$ is a closed and possibly disconnected surface.  Furthermore he allows $\bdry_+W$ to have a sphere component as long as it bounds a ball. See \cite{bonahon-cobordism}.
\end{remark}

\begin{remark}
A compression body $W$ may be regarded as a special kind of sutured manifold where $R_+=\bdry_+ W$ and $R_-=\bdry_- W$, and the annuli $\bdry_v W$ are the sutures.  
Furthermore, if $W$ is a connected compression body with  $\bdry W$ connected, then $W$ is also a handlebody if we forget the sutured manifold structure (that is, if we set $\bdry_+ W = \bdry W$).  Throughout, our compression bodies will retain their sutured manifold structure.
\end{remark}

\subsection{Heegaard splittings}

\begin{definition}[{\em Heegaard splittings}]\label{def:HS}
Let $M$ be a connected sutured manifold without toroidal sutures.   
A {\em Heegaard splitting} of $M$, denoted $(M,S)$, is a decomposition of $M$ along an oriented properly embedded surface $S$ into a pair of non-empty connected compression bodies $A$ and $B$ so that $M = A \cup B$ where $S = A \cap B = \bdry_+ A = \bdry_+ B$,  $R_- = \bdry_- A$, and $R_+ = \bdry_- B$.  (Here we flip the orientation of $B$ as a sutured manifold.)  The surface $S = A \cap B$ is the {\em Heegaard surface} of the splitting, and this surface defines the splitting.  For notational convenience we also refer to the Heegaard splitting as the triple $(A,B; S)$. Observe that the boundary components of $S$ may be regarded as the core curves of the annular sutures of $M$.      Further note that $S$ is necessarily connected.

A Heegaard splitting $(A, B;S)$ is {\em trivial} if either $A$ or $B$ is a trivial compression body.

Any connected sutured manifold $M$ without toroidal sutures has a Heegaard splitting as long as neither $R_+$ nor $R_-$ has a sphere component, cf. \cite{CG-RHS, Goda-SuturedManifold}.\end{definition}

\begin{definition}[{\em Circular Heegaard splitting}]
Let $M$ be a connected sutured manifold, possibly with toroidal sutures. 
A {\em circular Heegaard splitting} of $M$, denoted $(M, R, S)$, is a decomposition along a pair of disjoint, properly embedded, oriented surfaces $R$ and $S$ so that:
\begin{itemize}
\item $\bdry R$ is contained in the sutures of $M$,
\item $M \cut R$ is a connected sutured manifold without toroidal sutures, and 
\item $S$ is a Heegaard surface for $M\cut R$.
\end{itemize}
\end{definition}

\begin{definition}[{\em Generalized Heegaard splittings}]\label{defn:ghs}
	Let $M$ be a sutured manifold, possibly with toroidal sutures; $M$ need not be connected. 
A {\em generalized Heegaard splitting} of $M$, denoted $(M, \calR, \calS)$, is a decomposition of $M$ along a pair of disjoint, properly embedded, oriented surfaces $\calR$ and $\calS$  satisfying the following conditions: 
	\begin{itemize}
	\item $\bdry \calR$ is in the sutures of $M$;
	\item $\calR$  decomposes $M$ into a collection $M \cut \calR$ of connected sutured submanifolds without toroidal sutures; 
	\item  $\calS$ intersects each component of $M \cut \calR$ in a single Heegaard surface; and
	\item the compression bodies of $M \cut (\calR \cup \calS)$ may be labeled $A$ or $B$ 
	so that each component of $\calR \cup \calS$ meets both a $A$ compression body and a $B$ compression body.
	\end{itemize}
The surfaces $\calR$ and $\calS$ are commonly called ``thin'' and ``thick'' respectively, as are their components.
	Note that each component of  $\calR$ is both a component of $R_+(M_i)$ and a component of $R_-(M_j)$ for some components $M_i, M_j$ of $M \cut \calR$ (allowing $M_i=M_j$).  We may choose the labeling of the compression bodies of $M \cut (\calR \cup \calS)$ so that for any component $M_i$ of $M \cut \calR$, the $A$ compression body of its splitting by $\calS$ meets $R_-(M_i)$ and the $B$ compression body meets $R_+(M_i)$. 

Observe that if $\calR=\nil$, then the generalized Heegaard splitting $(M, \nil, \calS)$ may be regarded as a Heegaard splitting of $M$ even if $M$ is disconnected.  Indeed, when $\calR=\nil$ then $M \cut \calR=M$ so that $M$ necessarily has no toroidal sutures and $\calS$ intersects each connected component of $M$ in a connected Heegaard surface as in Definition~\ref{def:HS}.  As such, the surface $\calS$ divides $M$ into compression bodies $A$ and $B$ that each have exactly one connected component for each component of $M$. As with Heegaard splittings of connected manifolds, we may also refer to this splitting as the pair $(M, \calS)$ or the triple $(A, B;\calS)$. 

The adjective {\em generalized} indicates that $\calR$ may split $M$ into possibly more than one connected sutured submanifold. 
\end{definition}

\begin{remark}
Often in the literature, the term ``generalized Heegaard splitting'' allows for disconnected compression bodies in Heegaard splittings of potentially disconnected manifolds.  In these contexts, typically the thin and thick surfaces $\calR$ and $\calS$ are partitioned as regular levels of a Morse function of $M$ to either $\R$ or $S^1$  appearing alternately.   Between a consecutive pair of regular thin  levels of $\calR$ is a (possibly disconnected) manifold for which the intermediate regular thick level of $\calS$ restricts to a Heegaard surface on each component.    This Morse-theoretic version of generalized Heegaard splittings may be recovered from our kind of generalized Heegaard splitting by inserting pairs of trivial compression bodies as needed and taking their unions.   One may add the adjective {\em linear} or {\em circular} to the term ``generalized Heegaard splitting''  to emphasize that the splitting is associated to a Morse function to $\R$ or $S^1$ respectively.  Note that it is possible for the same generalized Heegaard splitting to be regarded as both linear and circular.  

Our approach to generalized Heegaard splittings falls more in line with that of \cite[Definition 4.1.7]{SSS-LectureNotes},  though we do not require exactness of the associated fork complex (see \cite[Definitions 4.1.3 \& 4.1.4]{SSS-LectureNotes}).  Furthermore,  we allow our compression bodies to have vertical boundary.
\end{remark}

\subsection{Weak Reductions and Amalgamations}
Explicit descriptions of weak reductions and amalgamations are given in \cite[\S2]{schultens-SfcexS1} in the context of compression bodies without vertical boundary. We present them here slightly differently while allowing vertical boundary, but the idea is basically the same.

\begin{definition}[{\em Stabilized splittings}]
Let $M$ be a connected sutured manifold without toroidal sutures.   Then both the Heegaard splitting $(A,B;S)$ and the Heegaard surface $S$ are called {\em stabilized} if $S$ has compressing disks $D_A$ in $A$ and $D_B$ in $B$ such that $\bdry D_A$ and $\bdry D_B$ are transverse in $S$ and intersect exactly once.  In such a situation, a compression of $S$ along $D_A$ (or $D_B$) produces a new Heegaard surface $S'$ of genus one less (allowing for the splitting of $S^3$ into two balls).  The reduction from $S$ to $S'$ is called a {\em destabilization}.  The inverse process (of tubing a Heegaard surface along a boundary parallel arc) is called a {\em stabilization}.
\end{definition}

\begin{definition}[{\em Reducible and irreducible splittings}]
Let $M$ be a connected sutured manifold without toroidal sutures.   Then both the Heegaard splitting $(A,B;S)$ and the Heegaard surface $S$ are called {\em reducible} if $S$ has compressing disks $D_A$ in $A$ and $D_B$ in $B$ such that $\bdry D_A=\bdry D_B$.  They are {\em irreducible} if they are not reducible.

Note that a stabilized splitting is reducible.    By Waldhausen \cite{waldhausen}, a reducible splitting of an irreducible manifold is stabilized.
\end{definition}

\begin{definition}[{\em Weakly reducible \& strongly irreducible}]
\label{defn:wr}
Let $M$ be a connected sutured manifold without toroidal sutures.   Then both the Heegaard splitting $(A,B;S)$ and the Heegaard surface $S$ are called {\em weakly reducible} if $S$ has compressing disks $D_A$ in $A$ and $D_B$ in $B$ that are disjoint.   They are called {\em strongly irreducible} otherwise.  That is, a Heegaard splitting $(A,B;S)$ is strongly irreducible if every compressing disk of $S$ in $A$ intersects every compressing disk of $S$ in $B$. 

If some component of $\calS$ of a generalized Heegaard splitting $(M, \calR, \calS)$  is a weakly reducible Heegaard surface for its component of $M \cut \calR$, then the generalized Heegaard splitting is {\em weakly reducible}. Otherwise, the generalized Heegaard splitting is {\em strongly irreducible}.  (In \cite[Definition 2.9]{SS-annuli}, it is further required that a strongly irreducible generalized Heegaard splitting has $\calR$ incompressible.  We call such a splitting locally thin, see Definition~\ref{defn:locallythin}.) 
\end{definition}

\begin{remark}
Vacuously, a trivial Heegaard splitting is strongly irreducible by definition.  In contrast, where the term was first introduced at the beginning of \cite[\S2]{CG-RHS}, Casson and Gordon required that a strongly irreducible splitting be also non-trivial.  Nevertheless, many authors define the term as we do (e.g.\  \cite[Definition 3.9]{Scharlemann-HSo3M}, \cite[Definition 3.1.14]{SSS-LectureNotes}).
At issue is how one treats the trivial splitting of a non-trivial compression body.  
\end{remark}

\begin{lemma}\label{lem:bdrycompisWR}
Let $M$ be a connected sutured manifold without toroidal sutures.  If $R_+ \cup R_-$ is compressible, then any non-trivial Heegaard splitting of $M$ is weakly reducible.  Consequently, if $(M, \calR, \calS)$ is a generalized Heegaard splitting in which $\calR \cup R_+ \cup R_-$ is compressible but no compression body of $M \cut (\calR \cup \calS)$ has a trivial component, then the splitting is weakly reducible.
\end{lemma}

\begin{proof}
The result follows from \cite[Theorem 2.1]{CG-RHS} and Definition~\ref{defn:wr}.
\end{proof}

\begin{remark}
Any weakly reducible Heegaard splitting may easily be ``inflated'' to a strongly irreducible generalized Heegaard splitting.  Let $S$ be a weakly reducible Heegaard surface for $M$, and consider a closed product neighborhood $S \times [-1,+1]$ of $S$ in $M$.  Setting $\calR = S \times \{0\}$ and $\calS = S\times \{+1\} \cup S \times \{-1\}$, the generalized Heegaard splitting $(M, \calR, \calS)$ is strongly irreducible as the components of $\calS$ give trivial splittings of the two compression bodies $M \cut \calR$.  Note that, however, $\calR$ is compressible.  (Such an ``inflation'' may be regarded as a trivial weak reduction, see Definition~\ref{defn:weakred}.)
\end{remark}

\begin{definition}[{{\em Locally thin}, cf.\ \cite[\S2]{MG-additivity} and \cite[Definition 3.2]{EMMGcircularthin}}]\label{defn:locallythin}
A generalized Heegaard splitting $(M, \calR, \calS)$ is {\em locally thin} if it is strongly irreducible and $\calR$ is incompressible.  
Note that we do not require $R_+ \cup R_-$ to be incompressible. 
\end{definition}

\begin{definition}[{\em Weak reduction}] \label{defn:weakred}
Let $(M, \calR, \calS)$ be a weakly reducible but irreducible generalized Heegaard splitting of an irreducible sutured manifold $M$.  Assume the component $S$ of $\calS$ defines an irreducible but weakly reducible Heegaard splitting $(A,B;S)$ of its component $M'$ of $M \cut \calR$.  A {\em weak reduction} of $(M, \calR, \calS)$ along $S$ is a generalized Heegaard splitting $(M, \calR', \calS')$ obtained as follows.

Let $D_A$ and $D_B$ be collections of disjoint compressing disks for $S$ in $A$ and $B$ respectively such that
	\begin{enumerate}
		\item $\bdry D_A$ and $\bdry D_B$ are disjoint in $S$, and
		\item\label{cond:2} no component of $S \cut (\bdry D_A \cup \bdry D_B)$ is a planar surface $P$ whose boundary is contained in either  only $\bdry D_A$ or only $\bdry D_B$ (possibly with duplicity).
	\end{enumerate}
	Consider a regular neighborhood $N = \nbhd(S \cup D_A \cup D_B)$ of $S$ and the disks $D_A$ and $D_B$ and the two submanifolds of $M'$ to either side of $N$, $A_1 = (M' \cut N) \cap A$ and $B_2 = (M' \cut N) \cap B$.  Within $N$ is the result $R$ of compressing $S$ along $D_A$ and $D_B$.   Then this surface $R$ splits $N$ into two (possibly non-connected) compression bodies $B_1$ and $A_2$ which meet $A_1$ and $B_2$ along surfaces $S_1$ and $S_2$ respectively.   Consequently $R$ splits $M'$ into two (possibly non-connected) sutured submanifolds $M'_1$ and $M'_2$ for which $(A_i, B_i; S_i)$ restricts to a Heegaard splitting on each component of $M'_i$ where $i=1,2$.  This gives the generalized Heegaard splitting $(M, \calR', \calS')$ in which $\calR' = \calR \cup R$ and $\calS' = (\calS - S) \cup (S_1 \cup S_2)$.

By the proof of \cite[Theorem 3.1]{CG-RHS} (see also its treatment in the proof of \cite[Theorem 3.11]{Scharlemann-HSo3M}), if the Heegaard surface $S$ is irreducible, then the disk sets $D_A$ and $D_B$ may be chosen so that the new surface $R$ resulting from the weak reduction is incompressible in $M$.  In such a case we say the weak reduction is {\em maximal}.  That is, $(M, \calR', \calS')$ results from a maximal weak reduction of $(M, \calS, \calR)$ along $S$.
\end{definition}

\begin{remark}\label{rem:incompfromwr}
	Observe that in Definition~\ref{defn:weakred}, the new component $R$ of $\calR'$  resulting from a maximal weak reduction along a weakly reducible but irreducible surface $S$ need not be connected.  Moreover, it is possible that one of the new components $S_1$ or $S_2$ of $\calS'$ produced by the maximal weak reduction is itself weakly reducible (cf.\ the paragraph after the proof of \cite[Theorem 3.11]{Scharlemann-HSo3M}), though it is necessarily irreducible. 
\end{remark}

\begin{remark}
In lieu of condition \ref{cond:2} in our Definition~\ref{defn:weakred}, some authors simply discard any sphere component that arises when compressing $S$ by $D_A \cup D_B$.  By our assumption of irreducibility of the manifold $M$, such sphere components either imply that $S$ is reducible or a disk of $D_A$ or $D_B$ is superfluous. 
\end{remark}

\begin{definition}[{\em Iterated weak reduction}]\label{defn:iwr}
If a generalized Heegaard splitting $(M,\calR, \calS)$ is weakly reducible but irreducible, then some component of $\calS$ is weakly reducible for its component of $M \cut \calR$.    Hence a weak reduction may be performed for that component.  Consequently we may have a sequence of weak reductions
	\[ (M, \calR, \calS)=(M, \calR_0, \calS_0)  \mapsto (M, \calR_1, \calS_1) \mapsto \cdots \mapsto (M, \calR_n, \calS_n)\]
	where each weak reduction $(M, \calR_{i}, \calS_{i}) \mapsto (M, \calR_{i+1}, \calS_{i+1})$ is performed along a weakly reducible component of $\calS_i$.  In particular, the weak reduction performed along the weakly reducible component $S$ of $\calS_i$ produces three surfaces $S_1$, $R$, and $S_2$ so that $\calR_{i+1} = \calR_i \cup R$ and $\calS_{i+1} = (\calS_i - S) \cup (S_1 \cup S_2)$. The resulting generalized Heegaard splitting $(M, \calR_n, \calS_n)$ is an {\em iterated weak reduction} of $(M, \calR, \calS)$.
	If no component of $\calS_n$ is weakly reducible and each weak reduction of the sequence is maximal, then  $(M, \calR_n, \calS_n)$ is a {\em maximal iterated weak reduction} of $(M, \calR, \calS)$.
\end{definition}

\begin{lemma}\label{lem:iteratedweakreduction2locallythin}
Let $M$ be a sutured manifold with irreducible Heegaard surface $S$.
Let $(M, \calR, \calS)$ be a maximal iterated weak reduction of $(M, \nil, S)$.  Then $(M, \calR, \calS)$ is a locally thin generalized Heegaard splitting.  
\end{lemma}

\begin{proof}
Since $(M, \calR, \calS)$ is the result of a maximal iterated weak reduction,  the components of $\calS$ are all strongly irreducible by Definition~\ref{defn:iwr}.  Hence the generalized Heegaard splitting is strongly irreducible.  Furthermore, as noted in Remark~\ref{rem:incompfromwr}, the components of $\calR$ produced by the maximal weak reductions are all incompressible.  Hence the generalized splitting is locally thin.
\end{proof}

\begin{definition}[{\em Amalgamation}]\label{defn:amalg}
The inverse process of a weak reduction is called an {\em amalgamation}.  
More specifically, we may speak of amalgamations of a generalized Heegaard splitting $(M, \calR, \calS)$ along certain unions of components of $\calR$.  
In addition to \cite[\S2]{schultens-SfcexS1}, see  \cite[\S3]{lackenby-HGalg} for another treatment of amalgamation. 

	Let $(M, \calR, \calS)$ be a generalized Heegaard splitting.  For $i=1,2$, let $M_i$ be two distinct connected components of $M \cut \calR$ with Heegaard splittings $(A_i, B_i; S_i)$ where $S_i$ is a component of $\calS$.  Assume the surface $R = B_1 \cap A_2 \subset \calR$ is non-empty.  ($R$ may be disconnected.)  
	We first define the {\em amalgamation along $R$} of  the generalized Heegaard splitting $(M_1  \cup_R M_2, R, S_1 \cup S_2)$  to be the Heegaard splitting $(M_1 \cup_R M_2, S')$ obtained as follows:
	
	\begin{adjustwidth}{.5cm}{.5cm}\indent
		 View the compression bodies $A_2$ and $B_1$ as being assembled from attaching $1$--handles to either side of a thickening $R \times [-1,+1]$ of $R$.  Note that we may assume that the feet of these $1$--handles have mutually disjoint projections to $R = R\times \{0\} \subset R \times [-1,+1]$.   Then we may extend all of these $1$--handles through the product $R \times [-1,+1]$ so that their feet now meet the other side.  In particular, this means that the feet of the $1$--handles of $A_2$ are in $S_1$ and the feet of the $1$--handles of $B_1$ are in $S_2$. Observe that $S_1$ tubed along these $1$--handles of $A_2$ is a connected surface that is also isotopic to $S_2$ tubed along these $1$--handles of $B_1$. In fact, extending the feet of the one-handles of both collections to only $R = R \times \{0\}$ and then tubing $R$ along these $1$--handles produces a surface $S'$ that is also isotopic to them.

	We now see that $S'$ is a Heegaard surface for $M_1 \cup_R M_2$.  View the compression body $A_1$ as also being assembled from $1$--handles attached to $\bdry_-A_1$ and a collection of $0$--handles. Then the feet of the $1$--handles from $A_2$ (now in $S_1$) can be slid off the $1$--handles of $A_1$ down to the surface $\bdry_- A_1$ and the $0$--handles.  Thus we have a new compression body $A'$ with $\bdry_- A' = \bdry_- A_1$ and $\bdry_+ A'= S'$.  Similarly we form a new compression body $B'$ with $\bdry_- B' = \bdry_- B_2$ and $\bdry_+ B' = S'$ by viewing $B_2$ as assembled from $1$--handles and sliding the $1$--handles from $B_1$ off them.  Hence we obtain the Heegaard splitting $(M_1 \cup_R M_2, S')$.
\end{adjustwidth}
	
	\noindent
	Now let $\calR' = \calR - R$ and $\calS' = \calS \cup S' -(S_1 \cup S_2)$.  Then we define the {\em amalgamation along $R$} of $(M, \calR, \calS)$ to be the generalized Heegaard splitting $(M, \calR', \calS')$.

	 Finally, an {\em amalgamation} of a generalized Heegaard splitting $(M, \calR, \calS)$ is the result of a sequence of amalgamations along surfaces as above. 
\end{definition}

\begin{remark}
	Continuing with the notation above, we may also consider amalgamations along $R$ when $R$ is only a proper subset of components of $B_1 \cap A_2$.  In this situation we may first ``inflate'' the generalized Heegaard splitting $(M, \calR, \calS)$ by inserting a product manifold with the trivial splitting for each component of $(B_1 \cap A_2) - R$ so that $B_1$ and $A_2$ meet only along $R$.  Then we take the amalgamation along $R$ of $(M, \calR, \calS)$ to be the amalgamation along $R$ of the ``inflated'' generalized Heegaard splitting. 
	While we will not need this extension of amalgamation,  it is interesting to note that it changes the homotopy class of the circular Morse function associated to the generalized Heegaard splitting.
\end{remark}

\subsection{Handle numbers and handle indices}

\begin{definition}[{\em Handle number}]\label{def:handlenumber}
	Following Goda \cite[Definition 3.4]{Goda}, the {\em handle number} of a compression body $W$ is the minimum number $h(W)$ of $1$--handles needed in its construction. In particular, if $W$ is connected, then 
	\[h(W) = g(\bdry_+ W) - g(\bdry_- W) + |\#\bdry_- W - 1|\]
	where $g(\bdry_- W)$ is the sum of the genera of the components of $\bdry_-W$ and $\#\bdry_- W$ is the number of its components.  (Set $g(\nil)=0$.)
	Viewing $W$ as built from attaching $1$--handles to $\bdry_-W$ and $0$--handles, one observes 
	\begin{itemize}
	\item exactly one $0$--handle is needed if $\bdry_-W = \nil$ (so that $W$ is a handlebody) and
	\item no $0$--handles are needed if $\bdry_- W \neq \nil$.
	\end{itemize}
	If $W$ is disconnected, then $h(W)$ is the sum of the handle numbers of its components.  

	Let $M$ be a sutured manifold.  
	The {\em handle number} $h(M, S)$ of a Heegaard splitting $(M,S)$ is the sum of the handle numbers of the compression bodies of $M \cut S$.  
	The {\em handle number} $h(M, \calR, \calS)$ of a generalized Heegaard splitting $(M, \calR, \calS)$ is the sum of the handle numbers of the compression bodies of $M \cut (\calR \cup \calS)$.

	The {\em handle number} $h(R)$ of a Seifert surface $R$ for an oriented link is the minimum of $h(M, R, S)=h(M\cut R, S)$ among circular Heegaard splittings $(M, R, S)$ of the link exterior $M$.  The {\em handle number} $h(L)$ of an oriented link $L$ is the minimum of $h(R)$ among Seifert surfaces of $L$ 	
	\begin{quote}
		{\sc Attention:} Our handle numbers $h(R)$ and $h(L)$ are {\em twice} what Goda defines.
	\end{quote}
\end{definition}

\begin{proposition}[{\cite[Proposition 3.7]{Goda}}]\label{prop:MNhandle}
If $L$ is an oriented link in $S^3$, then $\MN(L) = h(L)$. \qed
\end{proposition}

\begin{definition}[{\em Handle index}]\label{defn:handleindex}
The {\em handle index} $j(W)$ of a compression body $W$ built by attaching $1$--handles to $\bdry_- W$ and a collection of $0$--handles is the number
\[j(W) = \#(\mbox{$1$--handles}) -  \#(\mbox{$0$--handles}).\]
  Dually, if $W$ is built by attaching $2$--handles to $\bdry_+ W$ and filling a collection of resulting $2$--spheres with $3$--handles, then 
  \[j(W) = \#(\mbox{$2$--handles}) -  \#(\mbox{$3$--handles}).\]
Equivalently, the handle index $j(W)$ of a compression body $W$
is its handle number $h(W)$ minus the number of handlebodies in $W$.  

	Let $M$ be a sutured manifold.  The {\em handle index} $j(M, S)$ of a Heegaard splitting $(M,S)$ is the sum of the handle indices of the compression bodies of $M \cut S$.  The {\em handle index} $j(M, \calS, \calR)$ of a generalized Heegaard splitting $(M, \calS, \calR)$ is the sum of the handle indices of the compression bodies of $M \cut (\calR \cup \calS)$.  
\end{definition}

\begin{remark}
In \cite[Definition 2.1]{SS-comparing} and \cite[Definition 2.5]{SS-annuli}, Scharlemann and Schultens define the handle index of a (connected) compression body to be $J(W) = \chi(\bdry_-W) - \chi(\bdry_+ W)$. 
This index is {\em twice} our handle index.  That is, $J(W)=2j(W)$.   
\end{remark}

\begin{lemma}
	\label{lem:equivalentexchange}
	For generalized Heegaard splittings, the handle index
	is preserved by weak reductions and amalgamations.
\end{lemma}

\begin{proof}
While weak reductions and amalgamations may create or cancel $0$-- and $1$--handle pairs or $2$-- and $3$--handle pairs, these operations never create or cancel $1$-- and $2$--handle pairs.  Hence these operations preserve the handle index.
\end{proof}

\begin{lemma}\label{lem:handlenumber}
Any two generalized Heegaard splittings without handlebodies that are related by a sequence of weak reductions and amalgamations have the same handle number.
\end{lemma}

\begin{proof}
Since $h(W) = j(W)$ for a connected compression body $W$ that is not a handlebody, this is an immediate consequence of Lemma~\ref{lem:equivalentexchange}
\end{proof}

\section{Proofs of Lemma~\ref{lem:incompseifertsfce} and Theorem~\ref{thm:main} }\label{sec:proofs}

First we show that the handle number of a knot may be realized by an incompressible Seifert surface.

\begin{proof}[of Lemma~\ref{lem:incompseifertsfce}]
By Definition~\ref{def:handlenumber}, $K$ has a Seifert surface $R$ such that $h(K) = h(R)$.  If $R$ is incompressible, then we are done.  So assume $R$ is compressible.  Figure~\ref{fig:incompseifertsfce} provides a schematic of  the remainder of the proof in this case.

\begin{figure}
\centering
\includegraphics[width = \textwidth]{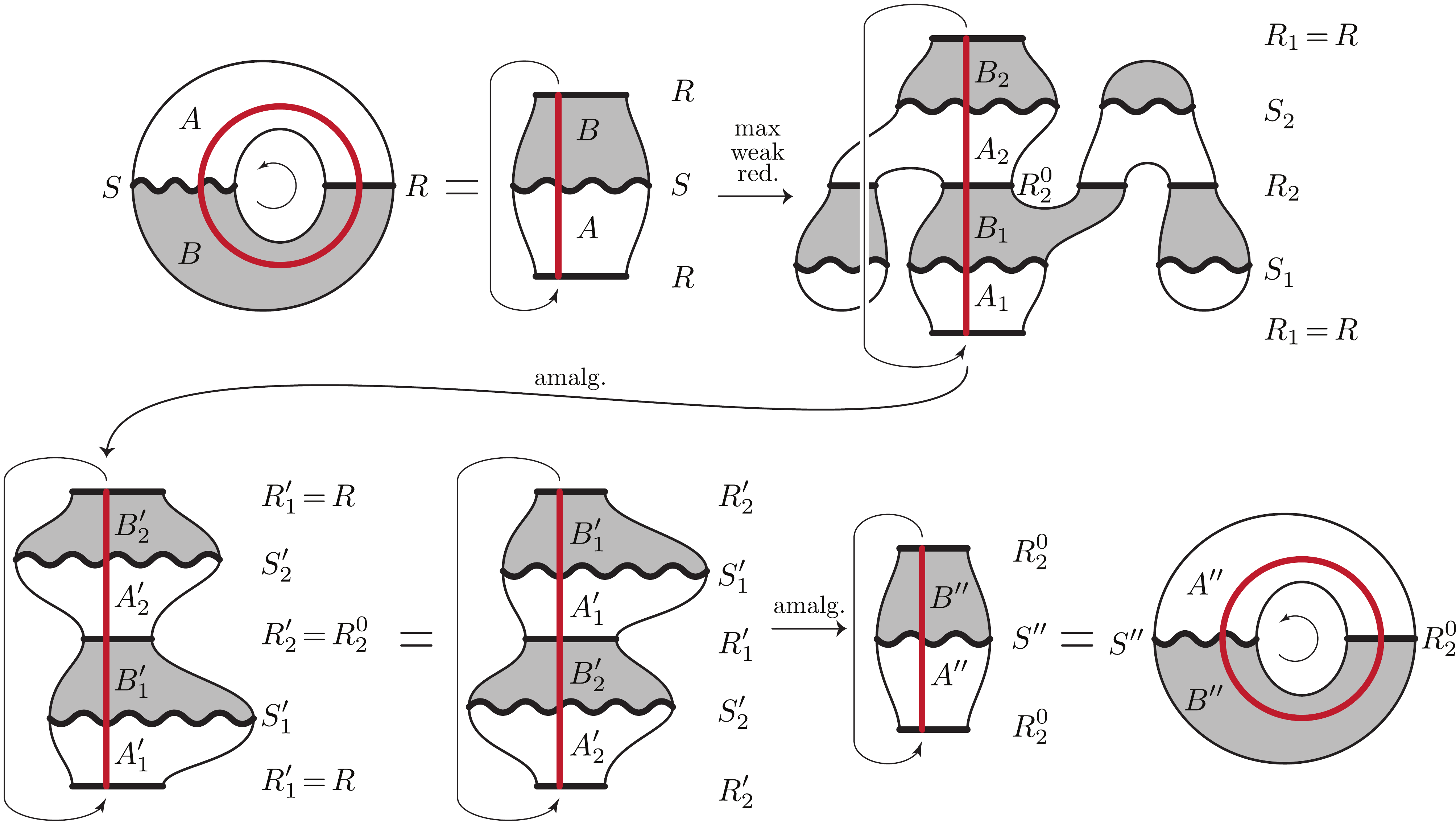}
\caption{Shown is a schematic of the sequence of a maximal weak reduction followed by amalgamations that transforms a circular Heegaard splitting $(M, R, S)$ with $R$ compressible into a circular Heegaard splitting $(M, R_2^0, S'')$ with $R_2^0$ incompressible and the same handle number.  The central loops/arcs (shown in red) represent the toroidal/annular sutures and the vertical boundary of the compression bodies.}
\label{fig:incompseifertsfce}
\end{figure}

Let $(M,R,S)$ be an associated circular Heegaard splitting of the knot exterior $M=S^3\cut \nbhd(K)$ realizing $h(R)$. 
That is, $h(M,R,S) = h(R)$.
Since $R$ is compressible,
this splitting is weakly reducible by Lemma~\ref{lem:bdrycompisWR}.  
(If the splitting had a trivial compression body, then both compression bodies of $M \cut (R \cup S)$ would be trivial.  Hence $R$ would be a fiber of a fibration of $M$ and thus be incompressible.) 
Therefore a maximal weak reduction of $S$ produces a generalized circular Heegaard splitting $(M, \calR, \calS)=(M, R_1 \cup R_2, S_1 \cup S_2)$ where $R=R_1$ and $R_2$ is an incompressible and possibly disconnected surface. 
Furthermore, we may view $R_2$ as dividing the connected sutured manifold $M \cut R_1$ into two (potentially disconnected) sutured manifolds $M_1$ and $M_2$ with Heegaard splittings $(A_1, B_1; S_1)$ and $(A_2, B_2; S_2)$ respectively such that $R_2 = B_1 \cap A_2$ and $R_1 = B_2 \cap A_1$.  Here, the compression bodies of these two splittings may be disconnected.  The compression bodies $B_1$ and $A_2$ have no handlebody components while the compression bodies $A_1$ and $B_2$ each consist of only handlebodies except for their component that meets $R_1$.   
(Following Definition~\ref{defn:ghs}, these two splittings of $M_1$ and $M_2$ are generalized Heegaard splittings.  For each $i=1,2$, on each component of $M_i$ the surface $S_i$ restricts to a connected Heegaard surface dividing that component of $M_i$ into a connected compression body component of $A_i$ and a connected compression body component of $B_i$.)

Since $S$ is a Seifert surface for $K$, it has a single boundary component.  Because the incompressible surface $R_2$ is obtained by compressing $S$, it also has a single boundary component.  Hence exactly one component of $R_2$, say $R_2^0$, has boundary and all other components are closed.  
In particular, $R_2^0$ is an incompressible Seifert surface for $K$.  
We will show that $h(K) = h(R_2^0)$. 


Because $R_1$ is just the Seifert surface $R$, the closed components of $\calR$ are the components of $R_2-R_2^0$.  Furthermore, observe that each of these closed components is null homologous in $M$, and so each bounds a submanifold of $M$.
As every component of $M \cut \calR$ is a component of $M_1$ or $M_2$, it
has a Heegaard splitting induced from the restriction of either $(A_1, B_1;S_1)$ or $(A_2, B_2;S_2)$.  Thus, except for the two meeting $\bdry M$, every component of $M \cut \calR$ must have
 one of the following forms:
\begin{enumerate}
	\item $A_1^i \cup_{S_1^i} B_1^i$ where $A_1^i$ is a handlebody component of $A_1$ and  $B_1^i$ is a connected component of the compression body $B_1$, or
	\item $A_2^i \cup_{S_2^i} B_2^i$ where $B_2^i$ is a handlebody component of $B_2$ and  $A_2^i$ is a connected component of the compression body $A_1$.
\end{enumerate}
The two components  of $M \cut \calR$ meeting $\bdry M$ must both meet $R_1$.  Thus they have induced Heegaard splittings in which one compression body has positive boundary $R_1$ and the other compression body has positive boundary containing $R_2^0$.
 
Perform amalgamations along the closed components of $R_2$ one at a time.  At each step the amalgamation is done along a remaining closed component of $R_2$ that is {\em innermost}, bounding a submanifold of $M$ that is disjoint from the other remaining components.  
Let us illustrate this using the schematic at
the top right of Figure~\ref{fig:incompseifertsfce}.  There we have depicted the splitting $(M, \calR, \calS)$ with $R_2$ as four horizontal segments in a row; call them $R_2^{\ell}$, $R_2^0$, $R_2^{m}$, and $R_2^{r}$ in order.  The segments $R_2^{\ell}$ and $R_2^r$ indicate innermost components of $R_2$.   Following Definition~\ref{defn:amalg}, an amalgamation along $R_2^r$ amalgamates the splittings of the components of $M_1$ and $M_2$ incident to it into a splitting of the manifold bounded by $R_2^m$ that consists of a black handlebody and a white compression body.  After this first amalgamation along $R_2^r$, the component $R_2^m$ is now innermost.  The next amalgamation may be performed along $R_2^\ell$ or $R_2^m$. Thereafter a third amalgamation on the final closed component of $R_2$ leaves us with what is depicted in the lower left of Figure~\ref{fig:incompseifertsfce}.
In general, this sequence of amalgamations results in a circular generalized  Heegaard splitting $(M, \calR', \calS')=(M, R_1' \cup R_2', S_1' \cup S_2')$ in which each surface $R_1'=R_1$, $R_2'=R_2^0$, $S_1'$, and $S_2'$ is connected and none is separating.  
(Since an amalgamation is the inverse process of a weak reduction, one may also obtain this splitting $(M, \calR', \calS')$ as a  weak reduction of the original splitting $(M, R, S)$.)

Now amalgamate $(M, \calR', \calS')$ along $R_1$ to obtain a splitting $(M, R_2^0, S'')$.  Since the surface $S''$ is an amalgamation of two connected surfaces, it too is connected.  
Since the  $(M, R_2^0, S'')$ and $(M, R, S)$ are splittings without handlebodies that are related by a sequence of weak reductions and amalgamations, they have the same handle number by Lemma~\ref{lem:handlenumber}. 
Hence $h(R_2^0)=h(M, R_2^0, S'') = h(M, R, S)=h(R)$.   
Therefore $R_2^0$ is an incompressible Seifert surface for $K$ such that $h(K)=h(R_2^0)$.
\end{proof}

Next we prove the additivity of Morse-Novikov number for knots by showing their handle number is additive.  For this we recall key results of Eudave-Mu\~noz and Manjarrez-Guti\'errez about locally thin circular generalized Heegaard splittings of exteriors of connected sums of knots.

\begin{theorem}[(Proposition 5.1 and Corollary 5.3 \cite{EMMGcircularthin})]\label{thm:EMMGsplit}
Let $M$ be the exterior of $K_a \# K_b$.  Let $Q \subset M$ be the properly embedded annulus that decomposes $M$ into the exteriors $M_a$ and $M_b$ of $K_a$ and $K_b$.  If $(M, \calR, \calS)$ is a locally thin circular generalized Heegaard splitting, then $\calR \cup \calS$ can be isotoped in $M$ so that
\begin{enumerate}
\item $Q \cut (\calR \cap \calS)$ is a collection of vertical rectangles in the compression bodies of the splitting, and
\item  the restriction $(M_i, \calR \cap M_i, \calS \cap M_i)$ is a circular generalized Heegaard splitting for each $i=a,b$. \qed
\end{enumerate}
\end{theorem}

\begin{proof}[of Theorem~\ref{thm:main}]
Using Goda's interpretation of the Morse-Novikov number in terms of handle number (see Proposition~\ref{prop:MNhandle}), we must show that $h(K_a \# K_b) = h(K_a) + h(K_b)$.  Let $M = S^3 \cut \nbhd(K)$ be the exterior of $K$, and let $M_i  = S^3 \cut \nbhd(K_i)$ be the exterior of $K_i$ for each $i=a,b$.

\medskip
The inequality $h(K_a \# K_b) \leq h(K_a) + h(K_b)$ is straightforward as implied in \cite[Section 14.6.2]{Pajitnov-CMT}.  Indeed for each $i=a,b$ let $R_i$ be a Seifert surface for $K_i$ such that $h(K_i)=h(R_i)$.  Then each knot exterior $M_i$ has a circular Heegaard splitting $(M_i, R_i, S_i)$ realizing its handle number.
The knot exterior $M$ may be obtained by gluing $M_a$ to $M_b$ along a closed regular annular neighborhood of a meridian in each of their boundaries.  This gluing may be done so that $\bdry R_a$ and $\bdry R_b$ meet along a single arc to form a boundary sum $R = R_a \natural R_b$ that is a Seifert surface for $K$.  Similarly $\bdry S_a$ and $\bdry S_b$ meet to form a boundary sum $S= S_a \natural S_b$.  Together they give a circular Heegaard splitting  $(M,R,S)$. From this we find that  $h(K) \leq h(R) = h(R_a)+h(R_b) = h(K_a) + h(K_b)$.

\medskip
For the other inequality, by Lemma~\ref{lem:incompseifertsfce} we may assume there is an incompressible Seifert surface $R$ for the knot $K=K_a \# K_b$ such that $h(K)=h(R)$.  Then there is an associated circular Heegaard splitting $(M, R, S)$ realizing this handle number. Because $S$ is connected, the two compression bodies of this splitting are connected and neither is a handlebody.  Let $(M, \calR, \calS)$ be a locally thin generalized circular Heegaard splitting resulting from a maximal iterated weak reduction as guaranteed by Lemma~\ref{lem:iteratedweakreduction2locallythin}.  Then
\[ h(K) = h(R) = h(M,R,S) = j(M,R,S) = j(M, \calR, \calS)\]
where the first two equalities are by definition, the third follows from Definition~\ref{defn:handleindex} because $(M,R,S)$ has no handlebodies in its splitting, and the fourth is due to Lemma~\ref{lem:equivalentexchange}.

Let $Q$ be the summing annulus that splits $M$ into the exteriors of $K_a$ and $K_b$. That is, $Q$ is the annulus along which $M_a$ and $M_b$ are identified to form $M$; $M \cut Q = M_a \sqcup M_b$. 
Since the splitting $(M, \calR, \calS)$ is locally thin, Theorem~\ref{thm:EMMGsplit} shows that we may arrange that $Q$ chops the splitting into circular generalized Heegaard splittings for $M_a$ and $M_b$.  In particular the restriction of $(M, \calR, \calS)$ to $M_i$ for $i=a,b$ is the splitting $(M_i, \calR_i, \calS_i)$ where $\calR_i = \calR \cap M_i$ and $\calS_i = \calS \cap M_i$.
Observe that a (not necessarily connected) compression body obtained by chopping another compression body along a vertical rectangle has the same handle index.  Hence 
\[j(M, \calR, \calS) =  j(M_a, \calR_a, \calS_a) + j(M_b, \calR_b, \calS_b).\]
For each $i=a,b$, a sequence of amalgamations brings $(M_i, \calR_i, \calS_i)$ to a circular Heegaard splitting $(M_i, R_i, S_i)$ with no handlebodies so that 
\[j(M_i, \calR_i, \calS_i) = j(M_i, R_i, S_i) = h(M_i,R_i,S_i) \geq h(R_i) \geq h(K_i).\]
(Note that if $(M_i, R_i, S_i)$ had a handlebody in its decomposition, then $S_i$ would have a closed component bounding this handlebody.  Then the component of the compression body meeting the other side of $S_i$ would have a closed component of $R_i$ in its negative boundary.  But then an amalgamation could be performed along this closed component.)
Thus we have 
\[h(K) \geq h(K_a) + h(K_b)\]
as desired.
\end{proof}

\section{Satellites and handle number}\label{sec:satellite}

\begin{definition}
	[Patterns and satellites]
	\label{defn:pattern}
An oriented two-component link $k \cup c$ in $S^3$ where $c$ is an unknot, defines a {\em pattern} $P$ which is the knot $k$ in the solid torus $S^3\cut \nbhd(c)$ equipped with a preferred longitude that is the meridian of $c$.  Then, given a knot $K$ and a pattern $P$, the {\em satellite knot} $P(K)$ is obtained by replacing a regular solid torus neighborhood of $K$ with $P$ so that the preferred longitudes agree. 

Let $P$ be a pattern defined by the link $k \cup c$.
The {\em winding number} of  $P$ is the linking number of $k$ with $c$.  We say $P$ is a {\em fibered pattern} if $k$ is a fibered knot and $c$ is transverse to the fibration.
\end{definition}

\begin{definition}[{\em Circular Heegaard splitting and handle number of a pattern}]
If the winding number of $P$ is non-zero, then we may define a {\em circular Heegaard splitting of $P$} to be a circular Heegaard splitting $(M,R,S)$ of the link exterior $M=S^3 \cut \nbhd(k \cup c)$ viewed as a sutured manifold with two toroidal sutures in which $R = \overline{R} \cut \nbhd(c)$ for a Seifert surface $\overline{R}$ of $k$.  In particular, $(M,R,S)$ is the restriction of a circular Heegaard splitting  $(\overline{M}, \overline{R}, \overline{S})$ of the knot exterior $\overline{M}=S^3 \cut \nbhd(k)$ where $c$ intersects the  two compression bodies of $\overline{M} \cut (\overline{R} \cup \overline{S})$ in unions of vertical arcs. 

Then we may take the {\em handle number} $h(P)$ of $P$ to be the minimum handle number among circular Heegaard splittings of $P$.
\end{definition}

The following is an immediate consequence of definitions.
\begin{lemma}\label{lem:fiberedpattern}
	For a pattern $P$ with non-zero winding number, $h(P)=0$ if and only if $P$ is fibered. \qed
\end{lemma}  

\begin{lemma}\label{lem:satellite}
	Let $P$ be a pattern with non-zero winding number, and let $K$ be a knot in $S^3$.  Then $h(P(K)) \leq h(P) + h(K)$.
\end{lemma}

\begin{proof}
We may build a generalized Heegaard splitting for $P(K)$ by gluing together circular Heegaard splitting for $P$ with an inflated circular Heegaard splitting for $K$. Then the handle number of the assembled generalized Heegaard splitting for $P(K)$ will be the sum of the handle numbers of the two circular Heegaard splittings.  The result will then follow by using circular Heegaard splittings that realize the handle numbers of $P$ and $K$.

Assume $P$ is defined by the link $k \cup c$ where $c$ is the unknot and the winding number is $n>0$.  Then the exterior $M = S^3 \cut \nbhd(P(K))$ of $P(K)$ is the union of $M_P=S^3\cut\nbhd(k \cup c)$ and $M_K = S^3 \cut \nbhd(K)$ along the tori $\bdry \nbhd(c)$ and $\bdry \nbhd(K)$ so that the meridian of $c$ is identified with the longitude of $K$ while the longitude of $c$ is identified with the reversed meridian of $K$.

Let $(M_P, R_P, S_P)$ be a circular Heegaard splitting for $P$. Then the thin and thick surfaces $R_P$ and $S_P$ chop $M_P$ into two connected compression bodies $A_P$ and $B_P$.  Each $A_P$ and $B_P$ has $n+1$ vertical boundary components; one is a longitudinal annulus of $\bdry \nbhd(k)$ while the other $n$ are meridional annuli of $\bdry \nbhd(c)$.  

Let $(M_K, R_K, S_K)$ be a circular Heegaard splitting for $K$. Then the thin and thick surfaces $R_K$ and $S_K$ chop $M_K$ into two connected compression bodies $A_K$ and $B_K$ whose vertical boundaries are each a single longitudinal annulus of $\bdry \nbhd(K)$.
Inflate this splitting $n-1$ times to produce the generalized Heegaard splitting $(M_K, \calR_K, \calS_K)$ where $\calR_K$ consists of $n$ parallel copies of $R_K$ and $\calS_K$ consists of $S_K$ and $n-1$ parallel copies of $R_K$ alternately between those of $\calR_K$. This inserts $2(n-1)$ trivial compression bodies (homeomorphic to $R_K \times [-1,1]$) 
into the circular Heegaard splitting $(M_K, R_K, S_K)$ at $R_K$. 

Now when we join $M_K$ to $M_P$ to form the exterior of $P(K)$, we may glue $(M_K, \calR_K, \calS_K)$ to $(M_P, R_P, S_P)$ so that $\bdry \calR_K$ identifies with $\bdry R_P \cap \bdry \nbhd(c)$ and $\bdry \calS_K$ identifies with $\bdry S_P \cap \bdry \nbhd(c)$. This produces Seifert surfaces for $P(K)$ where $R$ is $R_P$ with $n$ copies of $R_K$ attached along $\bdry R_P \cap \bdry \nbhd(c)$ and $S$ is $S_P$ with $S_K$ and $n-1$ copies of $R_K$ attached along $\bdry S_P \cap \bdry \nbhd(c)$.  Furthermore, this then causes the compression bodies $A_K$ and $B_K$ and the $2(n-1)$ trivial compression bodies of $M_K \cut(\calR_K \cup \calS_K)$ to be attached to $A_P$ and $B_P$ along their vertical boundaries that are the meridional annuli of $\bdry \nbhd(c)$. 
Since gluing two compression bodies $W_1$ and $W_2$ together along a vertical boundary component (so that $\bdry \bdry_+ W_1$ meets $\bdry \bdry_+ W_2)$ produces another compression body,  $(M, R, S)$ is a circular Heegaard splitting for $M$, the exterior of $P(K)$.  
 
By construction, $h(M,R,S) = h(M_P, R_P, S_P) + h(M_K, R_K, S_K)$.  Hence, choosing to use circular Heegaard splittings for which $h(M_P, R_P, S_P) = h(P)$ and $ h(M_K, R_K, S_K)= h(K)$, then we obtain $h(P(K)) \leq h(P) + h(K)$.
\end{proof}

\section{Cables and the proof of Theorem~\ref{thm:cable}}\label{sec:cables}

\begin{definition}[{\em Cables}]\label{defn:cable}
For coprime integers $p,q$ with $p>0$, let $k$ be a $(p,q)$--torus knot in the Heegaard torus $T$ of $S^3$. 
One of the core curves of the solid tori bounded by $T$ has linking number $p$ with $k$, and we let $c$ be that unknot. 
Then the link $k \cup c$ defines the {\em $(p,q)$--cable pattern} $P_{p,q}$.  For a knot $K$, its {\em $(p,q)$--cable} is the satellite knot $K_{p,q} = P_{p,q}(K)$.  The {\em cabling annulus} for the cabled knot $K_{p,q}$ is the image of the annulus $T-\nbhd(k)$ of $P_{p,q}-\nbhd(k)$ in the exterior $S^3\cut \nbhd(K_{p,q})$ of the cabled knot.
\end{definition}

\begin{lemma}
	\label{lem:fiberedcable}
	$P_{p,q}$ is a fibered pattern.
\end{lemma}
\begin{proof}
	This is rather well-known. Let $k \cup c$ be the link described in Definition~\ref{defn:cable}.  Then there is a Seifert fibration of $S^3$ in which $k$ is a regular fiber and $c$ is an exceptional fiber of order $p$. (If $p=1$ then $c$ may also be regarded as a regular fiber.)  This Seifert fibration restricts to a Seifert fibration on the exterior of $k$.  Since torus knots are fibered knots, the exterior of $k$ also fibers as a surface bundle over $S^1$.  As the Seifert fibration is transverse to this fibration, we see that $c$ is transverse to the fibration.  Hence $P_{p,q}$ is a fibered pattern.
\end{proof}

The proof of Theorem~\ref{thm:EMMGsplit} given in \cite{EMMGcircularthin} extends directly for any properly embedded essential annulus in a manifold with locally thin circular generalized Heegaard splitting as long as no boundary curve of the annulus is isotopic to a boundary curve of the thin surface.  We record this extension here as we will apply it with a cabling annulus. 
\begin{theorem}[({Cf.\ Theorem~\ref{thm:EMMGsplit} and \cite[Proposition 5.1 and Corollary 5.3]{EMMGcircularthin}})]
	\label{thm:annulus}
	Let $M$ be an irreducible sutured manifold with toroidal sutures.  Let $Q$ be a properly embedded essential annulus with $\bdry Q$ in the toroidal sutures of $M$.  
	If $(M, \calR, \calS)$ is a locally thin circular generalized Heegaard splitting such that no curve of $\bdry Q$ is isotopic in $\bdry M$ to a curve of $\bdry \calR$, then $\calR \cup \calS$ can be isotoped in $M$ so that 
	\begin{itemize}
		\item $Q \cap (\calR \cup \calS)$ is a collection of vertical rectangles in the compression bodies of the splitting, and 
		\item chopping along $Q$ renders $(M, \calR, \calS)$ into a circular generalized Heegaard splitting $(M \cut Q, \calR \cut Q, \calS \cut Q)$. \qed
	\end{itemize}
\end{theorem}
Note that, though we do not need it here, we have stated this theorem to allow for $Q$ to be non-separating and even for $\bdry Q$ to be contained in different components of $\bdry M$.

\begin{proof}[of Theorem~\ref{thm:cable}]
	Recall $p,q$ are coprime integers with $p>0$.  We aim to show that $h(K_{p,q})=h(K)$ for a knot $K$. By Goda (Proposition ~\ref{prop:MNhandle}), the corresponding statement for Morse-Novikov number will hold too.
	
	First off, if $p=1$ then $K_{p,q}=K$, and so the result is trivial.  Hence we assume $p\geq 2$.
	
	Since the pattern for a $(p,q)$--cable is a fibered pattern  by Lemma~\ref{lem:fiberedcable}, it follows from Lemmas~\ref{lem:fiberedpattern} and \ref{lem:satellite} that $h(K_{p,q}) \leq h(K)$.  Hence we must show $h(K_{p,q}) \geq h(K)$.

	Let $M=S^3 \cut \nbhd(K_{p,q})$ be the exterior  of the cabled knot $K_{p,q}$, let $M_K = S^3 \cut \nbhd(K)$, and let $M_V$ be a solid torus (the exterior of the unknot).  
	The cabling annulus $Q$ is a properly embedded annulus that chops $M$ into $M_K$ and $M_V$; $M\cut Q = M_K \sqcup M_V$.
	The proof now follows similarly to that of Theorem~\ref{thm:main}.  
	
	Let $(M,R,S)$ be a circular Heegaard splitting for the exterior of $K_{p,q}$ such that $h(K_{p,q}) = h(M,R,S)$.  By iterated maximal weak reductions we may obtain a locally thin generalized Heegaard splitting $(M,\calR, \calS)$ realizing this handle number as guaranteed by Lemma~\ref{lem:iteratedweakreduction2locallythin}.  Then 
	\[ h(K_{p,q}) = h(M,R,S)= j(M,R,S) = j(M,\calR, \calS).\]

	Since $p\geq 2$, the  boundary components of the cabling annulus $Q$ have slope distinct from the boundary slope of a Seifert surface\footnote{Indeed, in standard meridian-longitude coordinates, the components of $\bdry Q$ have slope $pq$.}. 
	We may now apply Theorem~\ref{thm:annulus} to obtain the generalized Heegaard splitting $(M\cut Q, \calR \cut Q, \calS \cut Q)$.  As in the proof of Theorem~\ref{thm:main}, it follows that 
	\[j(M, \calR, \calS) = j(M \cut Q, \calR \cut Q, \calS \cut Q).\]
	Because $M\cut Q = M_K \sqcup M_V$, we may divide  the generalized Heegaard splitting $(M\cut Q, \calR \cut Q, \calS \cut Q)$ into generalized Heegaard splittings $(M_K, \calR_K, \calS_K)$ and $(M_V, \calR_V, \calS_V)$.  
	Hence 
	\[j(M, \calR, \calS) = j(M_K, \calR_K, \calS_K) + j(M_V, \calR_V, \calS_V).\]
	For each $i=K,V$, a sequence of amalgamations along closed components of $\calR_i$ brings $(M_i, \calR_i, \calS_i)$ to a circular  Heegaard splitting $(M_i, R_i, S_i)$ with no handlebodies so that
	\[ j(M_i, \calR_i, \calS_i) = j(M_i, R_i, S_i) = h(M_i,  R_i, S_i).\]
	Then $h(M_K, R_K, S_K) \geq h(K)$ and $h(M_V, R_V, S_V) \geq 0$.
	Hence
	\[h(K_{p,q}) \geq h(K) + 0 = h(K)\]
	as desired.
\end{proof}

\begin{remark}
Alternatively, at the end of the above proof, instead of chopping along $Q$ once it is positioned to intersect the compression bodies of the generalized Heegaard splitting $(M, \calR, \calS)$ in only vertical rectangles,  one may prefer to first amalgamate to a circular Heegaard splitting $(M, R', S')$  while  preserving the nice structure of how $Q$ intersects the splittings.  Nevertheless, chopping this circular Heegaard splitting along $Q$ necessarily produces two generalized Heegaard splittings, and (unless $p=1$) amalgamations would still be needed to obtain circular Heegaard splittings.   

Indeed, when a Seifert surface for $K_{p,q}$ (such as an incompressible Seifert surface) may be isotoped to intersect $Q$ only in spanning arcs, it does so in $|pq|$ arcs.  Then $Q$ chops the surface into $p$ Seifert surfaces for $K$ and $|q|$ Seifert surfaces for the unknot. 
\end{remark}

\begin{acknowledgements}
We thank Andrei Pajitnov for introducing this problem to us during his visit to the University of Miami in Spring 2019.  We  thank  Jennifer Schultens for her correspondence and both Nikolai Saveliev and Chris Scaduto for prompting Question~\ref{ques:other} and the comparison with knot Floer homology.  We also thank Scott Taylor for conversations and valuable commentary on an earlier draft.  Finally, we thank the anonymous referee for their questions and comments that helped improve this article.
\end{acknowledgements}




\def\thispagestyle#1{}
\bibliographystyle{alpha}
\bibliography{MNnumber}

\affiliationone{
Kenneth L.\ Baker\\
Department of Mathematics, University of Miami\\
Coral Gables, FL 33146\\
 USA\\
\email{k.baker@math.miami.edu}
}
\end{document}